\documentclass[12pt]{amsart}
\usepackage[margin=1in]{geometry}
\usepackage{amsfonts}
\usepackage{amsmath}
\usepackage{amsthm}
\usepackage{amssymb}
\usepackage{tikz-cd}
\usepackage{mathrsfs} 
\usepackage{hyperref,url}
\usepackage{cleveref}
\usepackage{enumerate}
\usepackage{mathtools}
\hypersetup{colorlinks=false}

\newtheorem{theorem}{Theorem}
\newtheorem{lemma}[theorem]{Lemma}

\newtheorem{proposition}[theorem]{Proposition}

\newtheorem*{remark}{Remark}
\crefname{theorem}{Theorem}{Theorems}
\crefname{lemma}{Lemma}{Lemmas}
\crefname{corollary}{Corollary}{Corollaries}
\crefname{proposition}{Proposition}{Propositions}
\crefname{conjecture}{Conjecture}{Conjectures}
\theoremstyle{definition}

\numberwithin{theorem}{section}

\newcommand{\IC}{\mathrm{IC}}

\title{A Vanishing Theorem for Varieties with Finitely Many Solvable Group Orbits}
\author{Yiyu Wang}
\address{Department of Mathematics, University of Wisconsin - Madison, 480 Lincoln Drive, 213 Van Vleck Hall, Madison, WI 53706}
\email{yiyu.wang@wisc.edu}
\date{\today}

\begin{document}

\begin{abstract}
    We reprove and generalize a result proved by \cite{Yavin1992} that the intersection cohomology groups of a toric variety with coefficient in a nontrivial rank one local system vanish. We prove a similar vanishing result for a certain class of varieties on which a connected linear solvable group acts, including all spherical varieties.
\end{abstract}

\maketitle

\section{Introduction}
Intersection homology was first introduced by Goresky and Macpherson to fix the failure of Poincare duality and to achieve other important topological results like Lefschetz hyperplane theorem for singular spaces. The intersection cohomology groups of a toric variety played an important role in Stanley's proof of the final part of McMullen's $g$-conjecture. These groups also occurred when calculating the decomposition theorem for a proper toric map between toric varieties (see \cite{deCataldoMiglioriniMustata}). The intersection cohomology groups with twisted coefficient are essential in the theory of mixed Hodge module and the decomposition theorem.

In this paper, we will reprove and generalize the following theorem in \cite{Yavin1992}.
\begin{theorem}[\cite{Yavin1992}]\label{past_result}
    For a normal toric variety $X$, let $\mathbb{T}$ denote the affine torus, and $\mathbb{T}\to X$ denote the maximal torus orbit. Suppose $\mathcal{L}$ is a rank one nontrivial local system (nontrivial means not isomorphic to the constant local system) on $\mathbb{T}$, then the intersection cohomology group with coefficient in $\mathcal{L}$ vanishes, for any perversity $\bar{p}$:
    \begin{equation*}
        IH^*_{\bar{p}}(X,\mathcal{L})=0.
    \end{equation*}
\end{theorem}

We generalize this result to the case $X$ equipped with a connected linear solvable group $B$ action with finitely many orbits. Throughout this paper $B$ will always refer to a connected linear solvable group. The maximal $B$-orbit will be denoted by $U$, which is a dense open subset of $X$. By a theorem of Rosenlicht \cite{Rosenlicht1963} (see also \cite{Brion2021}), $U$ is isomorphic to a product of an affine torus and an affine space. Unfortunately, we are not able to prove the similar vanishing result for every nontrivial rank one local system on $U$, but almost every one. Fix a point $u\in U$, and let $p:B\to U$ denote the orbit map $p(b)=b\cdot u$. Call a rank one local system $\mathcal{L}$ \emph{twisted} if the pullback of $\mathcal{L}$ along the map $p$ is a nontrivial local system on $B$. Notice that this notion is independent of the choice of the point $u$. Our main result is the following.
\begin{theorem}\label{main_theorem}
    Let $X$ be a normal variety such that a linear connected solvable group $B$ acts on $X$ with finitely many orbits. Let $\mathcal{L}$ be a twisted rank one local system on the maximal $B$ orbit. Then for any perversity $\bar{p}$, the intersection cohomology group with coefficient in $\mathcal{L}$ vanishes:
    \begin{equation*}
        IH^*_{\bar{p}}(X,\mathcal{L})=0.
    \end{equation*}
\end{theorem}

If the map $p:B\to U$ has connected fibers, i.e. the stabilizer of a point in $U$ is connected, then a rank one local system on $U$ is twisted if and only if it is nontrivial, in particular \cref{main_theorem} is a generalization of \cref{past_result}. The class of normal varieties with finitely many solvable group orbits includes all spherical varieties, and the closure of any solvable group orbit on them, as long as they are normal.

Here is an outline of the proof. We will first define the notion of a ``twisted complex'', in an appropriate generality. We will prove that the property of being twisted is preserved under shifting and truncation (\cref{trivial_property}). The main result is that in our settings, a derived pushforward of a twisted complex is again twisted (\cref{reduction_to_derived_pushforward}). Put all these together, we can prove any intersection complex of any perversity is twisted if it is a twisted local system on an open subset. Using a spectral sequence, we can prove such intersection complexes have vanishing cohomology groups. To prove our main result, we reduce it to proving a derived pushforward of a twisted local system is again twisted (\cref{Main_tech_result}). This is proved by a careful analysis of the local structure of a $B$-variety, and we reduce our proof to a local computation.

This proof is new even in the case $X$ is a normal toric variety. The new idea is using compact torus rather than the affine torus. The use of compact torus enables us to construct equivariant tubular neighborhood easily. Our result can give a general formula for computing the decomposition theorem of a proper toric map between two toric varieties. This general version may also be used to compute the decomposition theorem of an equivariant map between two spherical varieties.

\subsubsection*{Acknowledgements.} The author would like to thank his advisor Botong Wang for his encouragement and various helpful discussions.

\section{Twisted local system}
We first describe a formulation that is useful in the proof of our main theorem.

A local system $\mathcal{L}$ on a compact torus $T=(S^1)^n$ are given by a representation of $\pi_1(T)=\mathbb{Z}^n$. By the Jordan-H\"older theorem, the local system $\mathcal{L}$ has a composition series
\begin{equation*}
    0\subset \mathcal{L}_1\subset \mathcal{L}_2\subset \cdots\subset \mathcal{L}_m=\mathcal{L}
\end{equation*}
such that each composition factor $\mathcal{L}_i/\mathcal{L}_{i-1}$ is a simple local system. On a compact torus $T=(S^1)^n$, any simple local system is of rank one, because $\pi_1(T)=\mathbb{Z}^n$ is abelian. A local system $\mathcal{L}$ on $T$ is called \emph{twisted} if all its composition factors $\mathcal{L}_i/\mathcal{L}_{i-1}$ are not isomorphic to the trivial local system $\mathbb{Q}_T$. The vanishing result we described in the introduction comes from the following standard result (see {\cite[Lemma 1.6]{Esnault1986}} for example).

\begin{lemma}\label{simple_vanishing_result}
    Let $T=(S^1)^n$ be a compact torus, $\mathcal{L}$ a twisted local system on $T$. Then the cohomology group $H^i(T,\mathcal{L})=0$ for all $i$.
\end{lemma}
\begin{proof}
    By the spectral sequence induced by the composition series, the proof is reduced to the case $\mathcal{L}$ is a nontrivial rank one local system, which is a standard result.
\end{proof}

Let us define the following notions. Suppose $X$ is a pseudomanifold equipped with a $T$-action. A stratification
\begin{equation*}
    \mathfrak{X}:X=X_n\supset X_{n-2}\supset\cdots\supset X_0\supset X_{-1}=\emptyset
\end{equation*}
is called \emph{$T$-stable} if every $X_i$ is $T$-stable, i.e. $T\cdot X_i\subset X_i$. We let $D^b_c(X)$ denote the bounded complexes of constructible sheaves on $X$ with coefficient in $\mathbb{C}$, and $D^b_{c,T}(X)$ denote the constructible sheaves with respect to some $T$-stable stratification. For a point $p\in X$, we let $\theta_p:T\to X$ denote the orbit map defined by $p$, $\theta_p(t)=t\cdot p$. We define
\begin{equation*}
    \mathcal{T}(X)\coloneqq\{\mathcal{F}^\cdot\in D^b_{c,T}(X): \forall p\in X, \theta_p^*(\mathcal{H}^i(\mathcal{F}^\cdot)) \text{ is a twisted local system} \}.
\end{equation*}
We call an element of $D^b_{c,T}(X)$ a $T$-constructible complex and an element of $\mathcal{T}(X)$ a twisted $T$-constructible complex. 
In particular, a local system is twisted if its pullback along any orbit map $\theta_p$ is a twisted local system on $T$. However, it is equivalent to require its pullback along \emph{some} orbit map is twisted, since for any two points $p,q$, and any path $\alpha:[0,1]\to X$ connecting $p$ and $q$, the map
\begin{equation*}
    H:T\times[0,1]\to X,\: H(t,x)=t\cdot \alpha(x)
\end{equation*}
gives a homotopy between $\theta_p$ and $\theta_q$. Therefore, the pullback local system does not depend on the choice of the point $p$.
\begin{remark}
    The condition that $\theta_p^*\mathcal{L}$ is twisted is slightly stronger than just requiring the restriction of $\mathcal{L}$ to the orbit of $p$ is twisted, because the stabilizer $T_p$ may be disconnected, in which case $\theta_p$ would be a composition of a finite covering map and a projection. The pullback of a twisted local system along a finite covering map may not be twisted.
\end{remark}

Let me first list a few properties of the collection $\mathcal{T}(X)$:
\begin{lemma}\label{trivial_property}
    For any pseudomanifold $X$ equipped with a $T$-action and a twisted constructible complex $\mathcal{F}^\cdot\in \mathcal{T}(X)$, we have
\begin{enumerate}
    \item Any truncation $\tau_{\leq n}\mathcal{F}^\cdot$ is in $\mathcal{T}(X)$.
    \item Any shift $\mathcal{F}^\cdot[m]$ is in $\mathcal{T}(X)$.
    \item (Two-out-of-three) For a distinguished triangle
    \begin{equation*}
        \mathcal{F'}^\cdot\to \mathcal{F}^\cdot\to \mathcal{F''}^\cdot\to\mathcal{F'}^\cdot[1]
    \end{equation*}
    in $D^b_{c,T}(X)$, if two out of three are in $\mathcal{T}(X)$, so is the third one.
\end{enumerate}    
\end{lemma}

\begin{proof}
    The first two properties are trivial. Let us prove the third one. The distinguished triangle induces a long exact sequence
    \begin{equation*}
        \cdots\to\mathcal{H}^i(\mathcal{F'}^\cdot)\to \mathcal{H}^i(\mathcal{F}^\cdot)\to \mathcal{H}^i(\mathcal{F''}^\cdot)\to\cdots
    \end{equation*}
    Without loss of generality, we assume $\mathcal{F'}^\cdot$ and $\mathcal{F''}^\cdot$ are in $\mathcal{T}(X)$. Then we pullback the long exact sequence along any orbit map $\theta_p$:
    \begin{equation*}
        \cdots\to\theta_p^*(\mathcal{H}^i(\mathcal{F'}^\cdot))\to\theta_p^* (\mathcal{H}^i(\mathcal{F}^\cdot))\to\theta_p^* (\mathcal{H}^i(\mathcal{F''}^\cdot))\to\cdots
    \end{equation*}
    which corresponds to a long exact sequence of $\pi_1(T)$-representations:
    \begin{equation*}
        \cdots\to M'^i\xrightarrow{p} M^i\xrightarrow{q}  M''^i\to\cdots,
    \end{equation*}
    Therefore
    \begin{equation*}
        0\to \mathrm{Im}\:p\to M^i\to \mathrm{Ker}\:q\to 0
    \end{equation*}
    is a short exact sequence. Since $\mathrm{Im}\:p$ and $\mathrm{Ker}\:q$ are quotient and subrepresentation of a twisted representation respectively, they are also twisted. We conclude that $M^i$ is also twisted and $\mathcal{F}^\cdot\in \mathcal{T}(X)$.
\end{proof}

As mentioned in the introduction, our main result is showing the collection of twisted complexes is stable under a derived open pushforward. The next proposition will reduce proving our main result to proving a derived open pushforward of a twisted local system is a twisted constructible complex.

\begin{proposition} \label{reduction_to_derived_pushforward}
    For a $T$-pseudomanifold $X$ equipped with a $T$-stable stratification $\mathfrak{X}$, suppose every stratum $X_\alpha$ has abelian fundamental group, and for any twisted rank one local system $\mathcal{L}$ on $X_\alpha$, $R(j_\alpha)_* \mathcal{L}$ is a twisted $T$-constructible complex, where $j_\alpha:X_\alpha\to X$ is the inclusion map. Then for any $\mathfrak{X}$-constructible open subset $j:U\to X$ of $X$, and any twisted $\mathfrak{X}$-constructible complex $\mathcal{F}^\cdot$ on $U$, the derived pushforward $Rj_*\mathcal{F}^\cdot$ is twisted.
\end{proposition}
\begin{proof}
    First notice that $R(j_\alpha)_* \mathcal{L}$ is a twisted $T$-constructible complex for a twisted $\mathcal{L}$ of \emph{any} rank. This is proved by inducting on the length of $\mathcal{L}$ and applying the two-out-of-three property. The base case that $\mathcal{L}$ is simple is exactly our assumption since $X_\alpha$ has abelian fundamental group which implies a simple local system must be of rank one. 

    We want to show that $Rj_*\mathcal{F}^\cdot$ is twisted. By doing truncation and two-out-of-three property, we can assume $\mathcal{F}^\cdot$ is a constructible sheaf (concentrated at degree 0).

    We do induction on $\dim X$. The base case is when $\dim X=0$, which means $X$ is a set of points. But there are no twisted local systems on a point. Therefore, the proposition is trivial in this case. Next we prove the general case. If $\mathcal{F}$ is supported on a lower dimensional closed subspace $Z\cap U\subset U$, where $Z\subset X$ is a closed subset, then this case is proved by the following diagram, and the proper base change theorem.
    \begin{center}
        \begin{tikzcd}
            U\cap Z \arrow[r, "i_U"] \arrow[d, "j_Z"] & U \arrow[d, "j"] \\
            Z \arrow[r, "i"]                          & X               
        \end{tikzcd}    
    \end{center}

    $$Rj_*\mathcal{F}^\cdot=Rj_*(i_U)_*i_U^*\mathcal{F}^\cdot=i_*R(j_Z)_*i_U^*\mathcal{F}^\cdot.$$ Since $Z$ satisfy all the assumptions on $X$, we are done by applying the inductive hypothesis to $U\cap Z\to Z$.
    
    If $\mathcal{F}$ is not supported on a lower dimensional subspace, then it is generically a local system.
    Let $V$ be the maximal open subset of $U$ that $\mathcal{F}$ restricts to a local system on each connected components of $V$ (notice that each component of $V$ is a stratum of $X$), and let $F=U\backslash V$. Then $\dim F<\dim X$. Consider the attaching triangle 
    \begin{equation*}
        (i_F)_*(i_F)^!\mathcal{F}^\cdot\to \mathcal{F}^\cdot\to R(j_V)_*(j_V)^*\mathcal{F}^\cdot\to (i_F)_*(i_F)^!\mathcal{F}^\cdot[1],
    \end{equation*}
    where $i_F:F\to U$ and $j_V:V\to U$. We apply $Rj_*$ on this triangle,
    \begin{equation*}
        Rj_*(i_F)_*(i_F)^!\mathcal{F}^\cdot\to Rj_*\mathcal{F}^\cdot\to Rj_*R(j_V)_*(j_V)^*\mathcal{F}^\cdot\to Rj_*(i_F)_*(i_F)^!\mathcal{F}^\cdot[1].
    \end{equation*}

    The complex $R(j_V)_*j_V^*\mathcal{F}^\cdot$ and $R(j\circ j_V)_*j_V^*\mathcal{F}^\cdot$ are twisted by out assumption because they are the pushforward of a twisted local system on strata. We conclude by two-out-of-three property that $(i_F)_*(i_F)^!\mathcal{F}^\cdot$ is twisted, and so is $Rj_*(i_F)_*(i_F)^!\mathcal{F}^\cdot$ by the inductive hypothesis. We conclude again by the two-out-of-three property that $Rj_*\mathcal{F}^\cdot$ is twisted.
\end{proof}

In the rest of the section, we will specialize to the case that $X$ is an algebraic variety equipped with a $B$-action with finitely many orbits on $X$, and the stratification $\mathfrak{X}$ is given by the $B$-orbits. The facts we need are summarized in the following lemma.

\begin{lemma}\label{auxillary_result}
    Let $\mathcal{O}\cong B/H$ be an orbit of $B$ and $\pi:B\to \mathcal{O}$ given by $\pi(b)= [bH]$. Fix a maximal torus $\mathbb{T}\subset B$, let $T\subset \mathbb{T}$ be the compact subtorus of $\mathbb{T}$. A rank one local system $\mathcal{L}$ on $\mathcal{O}$ is twisted under the $T$-action if and only if $\pi^*\mathcal{L}$ is nontrivial. The cohomology group $H^*(\mathcal{O},\mathcal{L})$ vanishes if $\mathcal{L}$ is twisted.
\end{lemma}

\begin{proof}
    Assume $\mathcal{L}$ is a rank one local system on $\mathcal{O}$. Since $B$ is homotopy equivalent to $T$, if $\pi^*\mathcal{L}$ is nontrivial, then $i^*\pi^*\mathcal{L}$ is nontrivial, where $i:T\to B$ is the inclusion map. But $i^*\pi^*=\theta_{[H]}^*$ by the very definition of $\theta_{[H]}^*$, so $\mathcal{L}$ is twisted.

    If $\mathcal{L}$ is a twisted local system on $\mathcal{O}$, then it is a successive extension of nontrivial rank one local systems on $\mathcal{O}$ by definition. By \cite{Rosenlicht1963}, $\mathcal{O}$ is isomorphic to $\mathbb{C}^n\times (\mathbb{C}^*)^n$, which is homotopy equivalent to a compact torus. The cohomology group $H^*(\mathcal{O},\mathcal{L})$ vanishes by \cref{simple_vanishing_result}.
\end{proof}

This lemma shows the definition of ``twisted" in the introduction coincides with our definition in this section. In particular, \emph{the notion of twistedness is independent of the choice of a maximal torus}.
Finally, we can prove in our setting, the hypercohomology of the dual of a twisted complex vanishes.

\begin{proposition}\label{spectral_argument}
    Suppose $X$ is an algebraic variety equipped with a $B$-action such that there are only finitely many orbits on $X$. Let $\mathfrak{X}$ be the stratification given by the orbits. If $\mathcal{F}^\cdot$ is constructible with respect to $\mathfrak{X}$, and its Verdier dual $D_X(\mathcal{F}^\cdot)$ is twisted, then
    \begin{equation*}
        H^*(X,\mathcal{F}^\cdot)=0.
    \end{equation*}
\end{proposition}
\begin{proof}
    Let $X_p$ be the union of orbits which are of dimension smaller or equal than $p$. Consider the spectral sequence (see \cite[Section 3.4]{Beilinson1996})
    \begin{equation*}
        E_1^{pq}=H^{p+q}(X_p\backslash X_{p-1},i_p^!\mathcal{F}^\cdot) \implies H^{p+q}(X,\mathcal{F}^\cdot),
    \end{equation*}
    where $i_p:X_p\backslash X_{p-1}\to X$ is the inclusion map. We only need to show that each $E_1^{pq}$ vanishes. We have
    \begin{equation*}
        i_p^!\mathcal{F}\cong i_p^!D_X^2(\mathcal{F}^\cdot)=D_{X_p\backslash X_{p-1}}(i_p^* D_X(\mathcal{F}^\cdot)),
    \end{equation*}
    and since $D_X(\mathcal{F}^\cdot)$ is twisted, $i_p^* D_X(\mathcal{F}^\cdot)$ is also twisted. Notice that the induced stratification is trivial on $X_p\backslash X_{p-1}$, so each cohomology sheaf of $i_p^* D_X(\mathcal{F}^\cdot)$ a twisted local system when restricting to each orbit in $X_p\backslash X_{p-1}$. By the same reason, taking cohomology sheaf commutes with the Verdier dual $D_{X_p\backslash X_{p-1}}$ up to a shift, so the cohomology sheaves of $D_{X_p\backslash X_{p-1}}(i_p^* D_X(\mathcal{F}^\cdot))$ are the dual of local systems occurred in cohomology sheaves of $i_p^* D_X(\mathcal{F}^\cdot)$, which are twisted. We conclude that $E_1^{pq}$ term vanishes by \cref{auxillary_result}.
\end{proof}

\begin{remark}
    This proposition can be rephrased in another way: if $\mathcal{F}^\cdot$ is twisted constructible, then $H^*_c(X, \mathcal{F}^\cdot)=0$. This gives yet another example that compactly support cohomology works better than the ordinary one when dealing with stratification.
\end{remark}

By Deligne's construction (see for example \cite[Section 6.3]{Maxim2019}), the intersection complexes can be constructed by doing successive pushforward and truncation. As a corollary of \cref{reduction_to_derived_pushforward} and \cref{spectral_argument}, \cref{main_theorem} is reduced to the following proposition, which will be proved in the next section. 
\begin{proposition}\label{key_proposition}
    Suppose $X$ is an irreducible algebraic variety equipped with a locally linear $B$ action, with finitely many $B$-orbits. The maximal $B$-orbit $U$ is a dense open subset of $X$. For any twisted rank one local system $\mathcal{L}$ on $U$, the derived pushforward $Rj_*\mathcal{L}$ is a twisted constructible complex.
\end{proposition}
Here locally linear means any $B$-orbit has an open $B$-stable neighborhood that can be $B$-equivariantly embedded into $\mathbf{P}(M)$, where $M$ is a representation of $B$, $\mathbf{P}(M)$ means the projective space of $M$. A normal $B$-variety is locally linear, by a fundamental result of Sumihiro \cite{Sumihiro1974}, and so is the closure of any $B$-orbit on it.

\section{Pushforward of a twisted local system is twisted}
We will prove a slightly more general version of \cref{key_proposition} in this section. I will first list a few lemmas that we will use during the proof. The reader can directly jump to the proof of \cref{Main_tech_result} and return to these lemmas when needed. We begin with the following well-known lemma. We will use this lemma to describe the local structure about the orbit.
\begin{lemma}[{\cite[Section 1.2]{Timashev2011}}]\label{homogeneous_fiber_bundle}
    Suppose $X$ is a topological space equipped with a topological group $G$ action, and there is a $G$-equivariant map $f:X\to G/H$, where $H$ is a subgroup of $G$. Then $X$ is $G$-equivariantly isomorphic to $G\times_H Y$, where $Y=f^{-1}([H])$ is an $H$-space, and $G\times_H Y\coloneqq (G\times Y)/H$, here $H$ acts on $G\times Y$ by the formula $h\cdot (g,y)=(gh^{-1},hy)$. Moreover, there is a one-to-one correspondence between $G$-orbits in $X$ and $H$-orbits in $Y$, and the stabilizer groups are the same under this correspondence.
\end{lemma}

\begin{lemma}\label{new_characterization_of_twistedness}
    Suppose $\mu$ is a finite subgroup of $T$, $A$ is a $\mu$-space. Let $c:T\times A\to T\times_\mu A$ be the map defined by quotient by $\mu$. Let $\mathcal{L}$ be a rank one local system on $T\times_\mu A$, and write $c^*\mathcal{L}=\mathcal{E}\boxtimes \mathcal{E'}$, where $\mathcal{E}$ is on $T$ and $\mathcal{E'}$ is on $A$. Then $\mathcal{L}$ is twisted if and only if $\mathcal{E}$ is twisted.
\end{lemma}
\begin{proof}
    Pick any point $p$ in $A$. We will denote the image of $(t,p)$ in $T\times_\mu A$ by $[t,p]$. The orbit map defined by the point $[e,p]$ is exactly the composition of $T\times \{p\}\hookrightarrow T\times A \to T\times_\mu A$ ($t\cdot [e,p]=[t,p]$). Therefore, $\mathcal{L}$ is twisted if and only if $c^*\mathcal{L}|_T=\mathcal{E}$ is twisted.     
\end{proof}

Suppose $X$ is a pseudomanifold equipped with a $T=(S^1)^n$ action. We call the action \emph{topologically locally linear} if for every $T$-orbit $\mathcal{O}\subset X$, there is a smooth $T$-manifold $N$ and a $T$-stable open neighborhood $V$ of $\mathcal{O}$, such that $V$ can be $T$-equivariantly embedded into $N$. In particular, being locally linear in \cref{key_proposition} implies being topologically locally linear here.

\begin{lemma}\label{local_linearity}
    Suppose $X$ is a pseudomanifold equipped with a topologically locally linear $T=(S^1)^n$ action, $\mathcal{O}=T\cdot p$ is an orbit. There is a $T$-stable open neighborhood $V$ of $\mathcal{O}$, which admits a $T$-equivariant retraction $\pi:V\to \mathcal{O}$.
\end{lemma}
\begin{proof}
    This is a corollary of the existence of the equivariant tubular neighborhood, which is a standard result (see for example \cite[Theorem 2.2]{Bredon1972296}). Since the action is topologically locally linear, we can find a $T$-stable open neighborhood $W$ of $T\cdot p$, such that $W$ can be $T$-equivariantly embedded into a smooth $T$-manifold $M$. We can find an equivariant tubular neighborhood $\widetilde{V}$ of $T\cdot p$ in $M$ with an equivariant retraction $\widetilde{\pi}:\widetilde{V}\to T\cdot p$. Let $V=\widetilde{V}\cap W$ and $\pi=\widetilde{\pi}|_V$, we conclude that there is an open $T$-invariant neighborhood of $T\cdot p$ in $X$ with an equivariant retraction $\pi:V\to T\cdot p$.
\end{proof}

\begin{theorem}\label{Main_tech_result}
    Suppose $X$ is an irreducible pseudomanifold equipped with a topologically locally linear $T=(S^1)^n$ action, and a $T$-stable stratification $\mathfrak{X}$. Let $j:U\to X$ be the maximal dimensional stratum, which we suppose admits a $T$-equivariant morphism to $T/\mu$, where $\mu\subset T$ is a finite subgroup. Then for any twisted rank one local system $\mathcal{L}$ on $U$, the derived pushforward $Rj_* \mathcal{L}$ is a twisted $T$-constructible complex.
\end{theorem}
\begin{proof}
    We must show that for any point $p\in X$, $\theta_p^*(R^qj_*\mathcal{L})$ is twisted for any number $q$. To calculate the derived pushforward, we want to find a good neighborhood of the orbit $T\cdot p$. This neighborhood is an analogy of the tubular neighborhood. There are two cases: either $\mathcal{L}$ has nontrivial monodromy around the orbit, in this case we expect the derived pushforward is zero; or $\mathcal{L}$ has some nontrivial monodromy in the orbit, in this case we expect the derived pushforward also has the same monodromy therefore nontrivial. The small loop that goes around the orbit is (in philosophy) given by the circle in the stabilizer group $T_p$, and the small loop that lives in the orbit should be given by the circle in the other summand that is isomorphic to quotient group $T/T_p$. However, in general $T_p$ is not connected, so there is no such splitting.
    
    Instead, we choose a splitting of the exact sequence
    \begin{equation*}
        1\to T_p^\circ\to {T}\to {T}/T_p^\circ\to 1,
    \end{equation*}
    where $\circ$ represents the identity component, and the subscription $p$ means the stabilizer group of $p$. We conclude that there is a subtorus $T'\cong T/T_p^\circ$ such that $T\cong T'\times T_p^\circ$. 

    For a twisted local system $\mathcal{L}$ on $U$, the pullback of $\mathcal{L}$ along any orbit map $T\to T\cdot x$ does not depend on the choice of $x$, and we will denote it by $\mathcal{L}|_T$. Since $\mathcal{L}|_T$ is of rank one, it can be written as a box product $\mathcal{L}|_T=\mathcal{L'}\boxtimes \mathcal{L''}$, where $\mathcal{L'}$ is on $T'$ and $\mathcal{L''}$ is on $T_p^\circ$. Since $\mathcal{L}|_T$ is twisted by our assumption, either $\mathcal{L'}$ or $\mathcal{L''}$ is twisted. This corresponds to the two cases we discussed before.

    If $\mathcal{L''}$ is twisted, then $T_p^\circ$ is not trivial which means there is some nontrivial monodromy around the orbit. In this case we expect that the derived pushforward vanishes at point $p$. Since $T_p^\circ$ fixes $p$, by \cref{auxillary_result}, we can find an open $T_p^\circ$-stable open neighborhood $V$ of $p$. The open subset $U$ is $T$-stable hence $T_p^\circ$-stable. By further mapping $T/\mu$ to $T/(\mu+T')\cong T_p^\circ/((T'+\mu)/T')$ (notice that this isomorphism is $T_p^\circ$-equivariant), we conclude from \cref{homogeneous_fiber_bundle} that $U\cong T_p^\circ\times_{\mu'} Y$ for some $\mu'$-space $Y$, where $\mu'=(T'+\mu)/T'$ is again a finite group. By \cref{homogeneous_fiber_bundle}, $U\cap V\cong T_p^\circ\times_{\mu'}Y'$, for some $\mu'$-stable subspace $Y'$ of $Y$. This is the desired local structure of the torus. It is nearly a product structure. Consider the map $\pi:T_p^\circ\times_{\mu'}Y'\to Y'/{\mu'}$, which is just the second projection module the action of $\mu'$. The fibers of this map is exactly the $T_p^\circ$-orbits in $U\cap V$. The derived pushforward $R\pi_*\mathcal{L}=0$ because its stalk at any point $y\in Y$ which corresponds to $T_p^\circ\cdot q$ is $$H^*(\pi^{-1}(y),\mathcal{L})=H^*(T_p^\circ\cdot q,\mathcal{L}).$$ Notice that the pullback of $\mathcal{L}|_{T_p^\circ\cdot q}$ by the map $T_p^\circ\to T_p^\circ\cdot q$ is exactly $(\mathcal{L}|_T)|_{T_p^\circ}=\mathcal{L''}$, which is twisted, hence the cohomology group vanishes by \cref{simple_vanishing_result}. Therefore, $$H^*(U\cap V, \mathcal{L})=H^*(Y'/\mu',R\pi_*\mathcal{L})=0.$$

    We conclude $$(Rj_* \mathcal{L})_p=\varinjlim_{p\in V}H^*(V\cap U, \mathcal{L})=0$$ by the fact that $T$-invariant open neighborhoods of $p$ form a local basis near $p$. Since $p$ can be chosen arbitrarily, we conclude that $Rj_* \mathcal{L}|_{T\cdot p}=0$ in this case hence twisted.

    If $\mathcal{L'}$ is twisted, then $T'$ is nontrivial, which means $\mathcal{L}$ is twisted in the sense of $T'$-action. This is the case when the nontrivial monodromy lives in the orbit. To calculate the derived pushforward, we must find a local structure. By \cref{local_linearity}, there is a $T$-stable open neighborhood $V$ of $T\cdot p$ in $X$ with an equivariant retraction $\pi:V\to T\cdot p$. The stabilizer of $T'$ action at $p$ is $T'\cap T_p$ which is a finite group since $T'\cap T_p^\circ=\{e\}$. Therefore, $T\cdot p=T'\cdot p\cong T'/T'_p$. 
    Apply \cref{homogeneous_fiber_bundle} to the morphism $\pi$, there is a $T'_p$-space $Y$ that $V\cong T'\times_{T'_p} Y$ $T'$-equivariantly. Since $U$ is $T'$-stable, there is a $T'_p$-stable open subset $Y'\subset Y$ such that $U\cap V\cong T'\times_{T'_p}Y'$ by \cref{homogeneous_fiber_bundle} again. We have now a Cartesian diagram.
    \begin{center}
        \begin{tikzcd}
            T'\times Y' \arrow[r, "c'"] \arrow[d, "j'"] & U\cap V\cong T'\times_{T'_p} Y' \arrow[d, "j"] \\
            T'\times Y \arrow[r, "c"]      & V\cong T'\times_{T'_p} Y           
        \end{tikzcd}
    \end{center}
     
    After taking a finite cover, we get a product structure, and the derived pushforward is easy to calculate in this neighborhood. The inclusion map $j'$ decomposes as $j'=\mathrm{Id}_{T'}\times j''$, where $j'':Y'\to Y$ is the inclusion. If we write $c'^*\mathcal{L}=\mathcal{E}\boxtimes \mathcal{E'}$, then by \cref{new_characterization_of_twistedness}, $\mathcal{E}$ is twisted. 
        
    We now calculate the derived pushforward by K\"unneth's formula:
    \begin{equation*}
            R^qj'_* c'^*\mathcal{L}\cong R^q(\mathrm{Id}_{T'}\times j'')_* (\mathcal{E}\boxtimes \mathcal{E'})=\mathcal{E}\boxtimes R^qj''_* \mathcal{E'}.
    \end{equation*}
    According to the proper base change theorem, $R^qj'_* c'^*\mathcal{L}\cong c^*R^qj_*\mathcal{L}$. Because $\mathcal{E}$ is twisted, $R^qj_*\mathcal{L}$ is twisted in the sense of $T'$-action by \cref{new_characterization_of_twistedness}. It follows directly from the definition that a rank one twisted local system in the sense of $T'$-action must be a twisted local system in the sense of $T$-action (pulling back along any $T$-orbit map then restricting to $T'$ is the same as pulling back along the $T'$-orbit map) and we are done.
\end{proof}

To prove \cref{key_proposition}, we only need to check the variety in that proposition satisfies every assumption in \cref{Main_tech_result}.

\begin{lemma}
    Let $\mathcal{O}=B/H$ be a homogeneous space under the $B$-action. There exists an affine subtorus $\mathbb{T}$ of $B$ such that for $T\subset \mathbb{T}$ its maximal compact torus  there is a $T$-equivariant map $p$ from $\mathcal{O}$ to $T/\mu$, where $\mu$ is a finite subgroup of $T$. For such torus $T$, a rank one local system $\mathcal{L}$ on $\mathcal{O}$ is twisted for $T$-action if and only if its pulling back to $B$ is nontrivial.
\end{lemma}
\begin{proof}
    Let $N$ be the maximal unipotent subgroup of $B$. Consider the projection $\pi:B/H\to B/NH$. Let $H^\circ$ denote the identity component of $H$, then we have a split short exact sequence
    \begin{equation*}
        1\to NH^\circ \to B\to B/NH^\circ\to 1.
    \end{equation*}
    It splits because $B/NH^\circ\cong (B/N)/(NH^\circ/N)$ is a quotient group of a torus hence also a torus, and $NH^\circ$ is connected. Therefore, we can find $\mathbb{T}\subset B$ such that $\mathbb{T}\cap NH^\circ=\{e\}$ and $\mathbb{T}\cong B/NH^\circ$. The projection $\pi$ is $B$-equivariant hence also $\mathbb{T}$-equivariant, and $\mathbb{T}$ acts on $B/NH$ transitively, with the stabilizer at $[NH]$ equals $\mathbb{T}\cap NH$ which is a finite subgroup $\mu\subset \mathbb{T}$. The map $\pi$ now becomes a $\mathbb{T}$-equivariant map from $\mathcal{O}$ to $\mathbb{T}/\mu$. Any finite subgroup must lie in the maximal compact subtorus $T$, so the quotient map $\mathbb{T}\to T$ induces a $T$-equivariant map $\mathbb{T}/\mu\to T/\mu$. By composing with $\pi$ we obtain the map we seek.

    Suppose $\mathcal{L}$ is a rank one local system on $\mathcal{O}$. By \cref{homogeneous_fiber_bundle}, there is a $T$-equivariant isomorphism $\mathcal{O}\cong T\times_\mu A$, where $A=p^{-1}([\mu])$. By our construction above, $[\mu]=[NH]$, $A\simeq NH/H\simeq N/(N\cap H)$, so $A$ is simply connected. Let $c:T\times A\to \mathcal{O}$, then $c^*\mathcal{L}=\mathcal{E}\boxtimes \mathbb{Q}_A$ since $A$ is simply connected. By \cref{new_characterization_of_twistedness}, $\mathcal{L}$ is twisted if and only $\mathcal{E}$ is twisted, in other words if and only if $c^*\mathcal{L}$ is nontrivial. Notice that the map $c$ is exactly $B/H^\circ\to B/H$. Since $B\to B/H^\circ$ has connected fiber, that $\mathcal{L}$ is twisted is also equivalent to that the pullback of $\mathcal{L}$ to $B$ is nontrivial.
\end{proof}

Therefore, \cref{key_proposition} follows from \cref{Main_tech_result}. Notice that although the conclusion of \cref{Main_tech_result} is that the pushforward of a rank one $T$-twisted local system is again $T$-twisted, the $T$-twistedness implies the pullback to $B$ is twisted. To be more clear: one starts with a twisted local system (in the sense of $B$-action) on $\mathcal{O}$, then chooses a $T$ using this lemma. The local system is $T$-twisted by this lemma. The \cref{Main_tech_result} tells us the derived pushforward is again $T$-twisted, and this implies it is twisted in the sense of $B$ action. Therefore, \cref{key_proposition} is proved. Now we are in a position to prove \cref{main_theorem}.

\begin{proof}[Proof of \cref{main_theorem}]
    By \cref{simple_vanishing_result}, we only need to show that the Verdier dual of an intersection complex with twisted coefficient $\IC_{\bar{p}}(\mathcal{L})$ is twisted. But its Verdier dual is $\IC_{\bar{q}}(\mathcal{L}^\vee)$ up to a shift, where $\bar{q}$ is the dual perversity, and $\mathcal{L}^\vee$ is the dual local system which is also twisted. Therefore, we need to show any intersection complex with twisted coefficient is twisted. By Deligne's construction, we have
    \begin{equation*}
        \IC_{\bar{p}}(\mathcal{L})=\tau_{\leq \bar{p}(n)-n}R{j_n}_*\cdots\tau_{\leq \bar{p}(2)-n}R{j_2}_*\mathcal{L}[n],
    \end{equation*}
    where $j_k:U_k\to U_{k+1}$ is the inclusion, and $U_k=X-X_{n-k}$. By \cref{reduction_to_derived_pushforward}, \cref{key_proposition}, the complex is twisted. We conclude from \cref{spectral_argument} that the hypercohomology group of $\IC_{\bar{p}}(\mathcal{L})$ vanishes, i.e. $IH^*_{\bar{p}}(X,\mathcal{L})=0.$
\end{proof}

\providecommand{\bysame}{\leavevmode\hbox to3em{\hrulefill}\thinspace}
\providecommand{\MR}{\relax\ifhmode\unskip\space\fi MR }
\providecommand{\MRhref}[2]{%
  \href{http://www.ams.org/mathscinet-getitem?mr=#1}{#2}
}
\providecommand{\href}[2]{#2}

\end{document}